\documentclass[11pt,reqno]{amsart}
\usepackage{amssymb,latexsym}
\usepackage{graphicx,xcolor}
\usepackage{url}		

\calclayout

\makeatletter
\newcommand{\monthyear}[1]{%
  \def\@monthyear{\uppercase{#1}}}
\newcommand{\volnumber}[1]{%
  \def\@volnumber{\uppercase{#1}}}
\AtBeginDocument{%
\def\ps@plain{\ps@empty
  \def\@oddfoot{\@monthyear \hfil \thepage}%
  \def\@evenfoot{\thepage \hfil \@volnumber}}
\def\ps@firstpage{\ps@plain}
\def\ps@headings{\ps@empty
  \def\@evenhead{%
    \setTrue{runhead}%
    \def\thanks{\protect\thanks@warning}%
    \uppercase{The Fibonacci Quarterly}\hfil}%
  \def\@oddhead{%
    \setTrue{runhead}%
    \def\thanks{\protect\thanks@warning}%
    \hfill\uppercase{Walking to infinity of the Fibonacci Sequence}}%
  \let\@mkboth\markboth
  \def\@evenfoot{%
    \thepage \hfil \@volnumber}%
  \def\@oddfoot{%
    \@monthyear \hfil \thepage}%
  }%
\footskip=25pt
\pagestyle{headings}%
}
\makeatother

\theoremstyle{plain}
\numberwithin{equation}{section}
\newtheorem{thm}{Theorem}[section]
\newtheorem{theorem}[thm]{Theorem}
\newtheorem{lemma}[thm]{Lemma}

\newtheorem{definition}[thm]{Definition}

\newtheorem{corollary}[thm]{Corollary}

\begin{document}
\monthyear{April 2022}
\volnumber{}
\setcounter{page}{1}

\title{Walking to Infinity on the Fibonacci Sequence}
\author{Steven J. Miller}
\address{Department of Mathematics and Statistics\\
        Williams College\\
        Williamstown, MA\\
        01267, USA}
\email{sjm1@williams.edu}
\author{Fei Peng}
\address{Department of Mathematics\\
        National University of Singapore\\
        119077, Singapore}
\email{pfpf@u.nus.edu}
\author{Tudor Popescu}
\address{Department of Mathematics\\
        Brandeis University\\
        Waltham, MA\\
        02453, USA}
\email{tudorpopescu@brandeis.edu}
\author{Nawapan Wattanawanichkul}
\address{Department of Mathematics\\
        University of Illinois Urbana-Champaign\\
        Urbana, IL\\
        61820, USA}
\email{nawapan2@illinois.edu}

\thanks{The authors would like to thank the other Polymath REU Walking to Infinity group members in summer 2020 for their contributions to the work. The group consisted of William Ball, Corey Beck, Aneri Brahmbhatt, Alec Critten, Michael Grantham, Matthew Hurley, Jay Kim, Junyi Huang, Bencheng Li, Tian Lingyu, Adam May, Saam Rasool, Daniel Sarnecki, Jia Shengyi, Ben Sherwin, Yiting Wang, Lara Wingard, Chen Xuqing, and Zheng Yuxi. We would also like to extend our gratitude to the referee for helpful comments on an earlier draft.}

\begin{abstract}
An interesting open problem in number theory asks whether it is possible to walk to infinity on primes, where each term in the sequence has one more digit than the previous. In this paper, we study its variation where we walk on the Fibonacci sequence. We prove that all walks starting with a Fibonacci number and the following terms are Fibonacci numbers obtained by appending exactly one digit at a time to the right have a length of at most two. In the more general case where we append at most a bounded number of digits each time, we give a formula for the length of the longest walk.
\vspace{-3pc}
\end{abstract}

\maketitle

\section{Introduction}

In our previous study \cite{M}, we examined an open problem that asks if it is possible to construct an infinite sequence, or as we call it, \textit{a walk to infinity}, on primes where each term, or \textit{step}, has one more digit than the previous. If we start from a single-digit prime, this problem is the same as finding a \textit{right truncatable prime}, which is a prime that remains prime after removing the rightmost digits successively. It is known that the largest right truncatable prime is 73939133 \cite{AG}, so one cannot walk to infinity starting with a one-digit prime, i.e., the longest such walk is
$$\{7, 73, 739, 7393, 73939, 739391, 7393913, 73939133\}.$$

However, the problem remains open if we allow ourselves to start with a prime of any digits. We \cite{M} showed that it is impossible to construct such a walk on primes in bases 2, 3, 4, 5, and 6. For primes in base 10, although we could not solve the problem, we gave stochastic models to determine the expected value of the length of such walks. We further studied the same question for other number-theoretical sequences, namely, perfect squares and square-frees. Moreover, motivated by the \textit{Gaussian moat problem} \cite{GWW}, which asks whether one can walk to infinity on primes in the quadratic integer ring $\mathbb{Z}[i]$ with a bounded step size, we investigated walks on primes in $\mathbb{Z}[\sqrt{2}]$ and gave a conditional proof that such walks do not exist \cite{L}.

In this paper, we shift our attention to the well-known Fibonacci sequence and, again, ask \textit{whether one can walk to infinity on the Fibonacci sequence by appending one digit to the right each time}. The definition of the Fibonacci sequence \cite{Fibo} can be given as follows:
\begin{definition}[Fibonacci numbers]\label{def:Fibo} Let $F_n$ be the \textit{$n$-th Fibonacci number} and $F_0 = 0, F_1 =1$. For $n \ge 2$, $$F_{n}=F_{n-1}+F_{n-2}.$$
\end{definition}
Using the above recurrence relation, one can approximate the ratio of any consecutive Fibonacci numbers: let $\lim_{n\to\infty}F_{n}/F_{n-1} = x$. Dividing the recurrence relation by $F_{n-1}$ and taking $n$ to infinity gives
$$\lim_{n\to \infty}\frac{F_n}{F_{n-1}} = 1 + \lim_{n\to \infty}\frac{F_{n-2}}{F_{n-1}},$$
which implies $x= 1 +1/x$ hence $x \approx 1.618\ldots$, i.e., the golden ratio. In other words, the Fibonacci numbers grow exponentially while the growth of primes is linear times logarithm, i.e., the $n$-th prime is asymptotic to $n\log n$ resulting from the prime number theorem. Unlike in the case of primes, this significantly more rapid growth of the Fibonacci numbers allows us to resolve the problem, and the result can be stated as follows:

\begin{theorem} \label{thm:fibo1digit} It is impossible to construct an infinite walk on the Fibonacci sequence by appending exactly one digit at a time to the right. In particular, all such walks have a length of at most 2.
\end{theorem}

Furthermore, we generalize this problem to appending \textit{at most} $N$ digit to the right each time and obtain the following theorem.

\begin{theorem}\label{thm:FiboAtmostN} Given we append at most $N$ digits to the right each time and the starting number contains $N_0 \ge 2$ digits, the length of the longest walk is then at most $\lfloor \log_2{\frac{N}{N_0-1}}\rfloor+2$. If $N_0=1$, the length of the longest walk is at most $\lfloor \log_2{N}\rfloor+2$.
\end{theorem}

\section{Proof of the Theorems \ref{thm:fibo1digit} and \ref{thm:FiboAtmostN}}

To prove our theorems, we first establish some relations between any two Fibonacci numbers with $k$ order apart, i.e., $F_m$ and $F_{m+k}$.

\begin{lemma}\label{lem:Fkbound} For all $m,k \in \mathbb{N}$, $$F_{k+1}F_{m} \ \le \ F_{m+k} \ \le \ F_{k+2}F_m.$$
\end{lemma}
\begin{proof} Let $m$ be any positive integer. We show that the statement is true for all $k \in \mathbb{N}$ by strong induction. For the base cases $k=1$ and $k=2$, by the recurrence relation in Definition \ref{def:Fibo} and the fact that the Fibonacci sequence is increasing, $F_{m} \le F_{m+1} \le 2F_{m}$ holds. Moreover, it follows that $2F_{m} \le F_{m+2} \le 3F_{m}$ by adding $F_m$ throughout the prior inequality.

Now, suppose that for all $k$ with $2 \le k \le r$, $F_{k+1}F_{m} \le F_{m+k} \le F_{k}F_m$. Taking $k = r-1$ and $r$ gives us
\begin{equation*} F_{r}F_{m} \ \le \ F_{m+r-1} \ \le \ F_{r+1}F_m
\end{equation*}
and
\begin{equation*}
F_{r+1}F_{m} \ \le \ F_{m+r} \ \le \ F_{r+2}F_m \end{equation*}
respectively. Thus, combining the above two inequalities yields
$$ F_{r-1}F_{m} + F_{r}F_{m}  \ \le \ F_{m+r-1}+ F_{m+r}  \ \le \  F_{r}F_{m} + F_{r+1}F_{m}.$$  Again, by the recurrence relation, the above inequality is equivalent to the lemma statement.
\end{proof}

\begin{lemma}\label{lem:Fkvalues} For all $m \ge k \in \mathbb{N}, k \ge 2$,  $$F_{m+k} = (F_{k+2} - F_{k-2})F_m + (-1)^{k+1}F_{m-k}.$$
\end{lemma}
\begin{proof}
Let $\varphi = (1 + \sqrt{5})/2$, the golden ratio. Using the supposition, we compute that
\begin{equation}\varphi^4 - 1 \ = \ \sqrt{5}\varphi^2. \label{eq:goldenration} \end{equation}
Furthermore, Binet's formula gives
\begin{equation} F_{n} \ = \ \frac{\varphi^n - (-\varphi)^{-n}}{\sqrt{5}},\label{eq:binet} \end{equation}
and hence
\begin{align*}
&(F_{k+2} - F_{k - 2})F_{m} + (-1)^{k + 1} F_{m-k}\\
& = \frac{(\varphi^{k+2} - (-\varphi)^{-k -2})-(\varphi^{k-2} - (-\varphi)^{-k+2}) }{\sqrt{5}}\cdot\frac{(\varphi^m - (-\varphi)^{-m})}{\sqrt{5}} + (-1)^{k + 1}\frac{(\varphi^{m - k} - (-\varphi)^{k - m})}{\sqrt{5}}\\
& = \frac{(\varphi^4 - 1)(\varphi^{k-2} + (-\varphi)^{-k - 2})}{\sqrt{5}}\cdot\frac{(\varphi^m - (-\varphi)^{-m})}{\sqrt{5}} + (-1)^{k + 1}\frac{(\varphi^{m - k} - (-\varphi)^{k - m})}{\sqrt{5}}.
\end{align*}
Thus, by \eqref{eq:goldenration}, we substitute $\varphi^4 - 1$ by $\sqrt5\varphi^2$ in the above expression and obtain
\begin{align*}
&\frac{\varphi^2(\varphi^{k-2} + (-\varphi)^{-k - 2})(\varphi^m - (-\varphi)^{-m})}{\sqrt{5}} + (-1)^{k + 1}\frac{(\varphi^{m - k} - (-\varphi)^{k - m})}{\sqrt{5}} \\
&= \frac{\varphi^{m + k} - (-\varphi)^{-m - k} + \varphi^{m - k} (-1)^{-k} - \varphi^{k - m} (-1)^{-m}}{\sqrt{5}} - \frac{(-1)^{k}\varphi^{m - k} - (-1)^{m} \varphi^{k - m}}{\sqrt{5}}\\
&= \frac{\varphi^{m + k} - (-\varphi)^{-m - k}}{\sqrt{5}},\end{align*}
which is exactly $F_{m+k}$ by \eqref{eq:binet}. \end{proof}

Now that we have Lemmas \ref{lem:Fkbound} and \ref{lem:Fkvalues}, we obtain our first theorem where we append exactly one digit at a time to the right.

\begin{proof}[Proof of Theorem \ref{thm:fibo1digit}] Starting with some Fibonacci number $F_m \ge 1$, if we append $d \in \{0,1,2, \dots, 9\}$, the newly appended number is then $10F_m + d$. From Lemmas \ref{lem:Fkbound} and \ref{lem:Fkvalues} respectively, the following statements hold.
\begin{equation} 5F_m  \le \ F_{m+4} \ \le \ 8F_m  \le \ F_{m+5} \ \le \ 13F_m \ \le \ F_{m+6} \  \le  21F_m, \label{eq:boundm+5} \end{equation}
and
\begin{equation} \text{For all,} \ m > 5 \in \mathbb{N}, \ F_{m+5} = 11F_m+F_{m-5}. \label{eq:valuesm+5}\end{equation}

Since $10F_m + d$ is a Fibonacci number and $8F_m+d \le 10F_m+d \le 13F_m +d$, \eqref{eq:boundm+5} implies that $10F_m + d$ is either $F_{m+5}$, or $F_{m+6}$ if $m$ is small.

\begin{itemize}
    \item If $10F_m + d = F_{m+6}$, we have that $10F_m + d \ge 13F_m$ and, thus, $d \ge 3F_m$. Since $d$ is a single-digit number, the possible values of $F_m$ are 1, 2 and 3.
    \item If $10F_m + d = F_{m+5}$, \eqref{eq:valuesm+5} tells us that if $m > 5$, $10F_m + d  = 11F_m +F_{m-5}$, so $d  = F_m +F_{m-5}$. Again, since $0 \le d \le 9$, $F_m \le 9$, and, as $m > 5$, the only possible value is $F_m =8$. Otherwise when $m \le 5$, $F_m$ is either $1, 2, 3, 5$ of $8$.
\end{itemize}
From both cases, we conclude that any walks must start from $F_m = 1, 2, 3, 5$, or $8$. This fact implies that there are only 5 possible walks of length 2, namely, $1 \rightarrow 13$, $2 \rightarrow 21$, $3 \rightarrow 34$, $5 \rightarrow 55$, and $8 \rightarrow 89$.
\end{proof}

Now, we apply the same technique to the case where we appended \textit{exactly} $N$ digits at a time to the right. By appending exactly $N$ digits, we also include appending $0$'s leading numbers, for example, $001$ or $0000002123$. This is an intermediate step towards Theorem \ref{thm:FiboAtmostN} where we append \textit{at most} $N$ digits each time. We first start from the following lemma.

\begin{lemma} \label{lem:10Nimpossible}
For all natural numbers $N$, there exists no natural number $k \ge 2$ such that $$F_{k+2}-F_{k-2} = 10^N.$$
\end{lemma}
\begin{proof} By observing the first 100 Fibonacci numbers \cite[Appendix~2, p.~585--588]{Kos}, we notice that $F_{61} \equiv F_{1}$ (mod $10$) and $F_{62} \equiv F_{2}$ (mod $10$). We show inductively that for any positive integer $n$, $F_{60+n} \equiv F_{n} \pmod{10}$. The case when $n=1$ or $n=2$ is established. As for the inductive step, if $F_{60+k} \equiv F_{k} \pmod{10}$ and $F_{60+(k+1)} \equiv F_{k+1} \pmod{10}$, then we have \[F_{60+(k+2)} = F_{60+k}+F_{60+(k+1)} \equiv F_{k}+F_{k+1} = F_{k+2} \pmod{10}.\]
The periodic property $F_{60+n} \equiv F_{n} \pmod{10}$ tells us that if there exists no pair of Fibonacci numbers $F_{k+2}$ and $F_{k-2}$, where $2 \le k \le 62$, such that $F_{k+2} - F_{k-2} \equiv 0$ (mod $10$), then there is no $k$ such that $F_{k+2}-F_{k-2} \equiv 0$ (mod $10$); as a result, it is impossible to have $F_{k+2}-F_{k-2} = 10^N$. This is because if there is no such a pair when $2 \le k \le 62$, neither does when $2+60m \le k \le 62+60m$, where $m \in \mathbb{N}$. Hence there are no possible pairs for any integer $k > 2$.

Going through the list of Fibonacci numbers again, we find no such pair in the first 62 Fibonacci numbers, which completes our proof.
\end{proof}

Lemma \ref{lem:10Nimpossible} then serves as an important tool to draw some conclusions in the next Lemma.

\begin{lemma}\label{lem:fiboNdigit} It is impossible to construct an infinite walk on the Fibonacci sequence by appending exactly $N$ digits at a time, where $N$ is a fixed positive integer. In particular, any appendable step in the walk must be of length at most $8/7\cdot (10^N-1)$.

(Note that, in this case, an appendable step refers to a step in a walk that we can append some $N$-digit number to the right and still get a Fibonacci. When a step is not appendable, the walk terminates.) \end{lemma}

\begin{proof} Let $N$ be a fixed positive integer and $F_m$ be the starting number of a walk. Similar to Theorem \ref{thm:fibo1digit} the next step in the walk can be written as \[10^NF_m+d, \text{ where } 0 \le d \le 10^N-1. \]
Now, let $k$ be such that $F_{k+1} \le 10^N$ and $F_{k+2} > 10^N$. By Lemma \ref{lem:Fkbound}, we have that
\begin{equation} F_{k+1}F_m \ \le \ F_{m+k} \ \le \ F_{k+2}F_m. \label{eq:boundN}
\end{equation}

Again, while the most likely case is when $10^NF_m+d = F_{m+k}$, there are two unlikely cases: $10^NF_m+d < F_{k+1}F_{m}$ and $10^NF_m+d > F_{k+2}F_{m}$.

If $10^NF_m+d < F_{k+1}F_{m}$, then \eqref{eq:boundN} fails because $F_{k+1} \le 10^N$ and $d$ is positive.

Now consider the case when $10^NF_m+d > F_{k+2}F_{m}$. Since $F_{k+2} > 10^N$ and hence $F_{k+2} \ge 10^N+1$, we have that $$ 10^NF_m +d \ > \ (10^N+1)F_m, $$ meaning that any appendable $F_m$ in the walk must be $\le d \le 10^N-1$.

 Lastly, for the most likely case when $10^NF_m +d = F_{m+k}$, from Lemma \ref{lem:Fkvalues}, we have \begin{align}\label{eq:Ndigit}
    10^NF_m +d  \ &= \ (F_{k+2}-F_{k-2})F_{m}+(-1)^{k+1}F_{m-k}\nonumber\\
    d \ &= \ (F_{k+2}-F_{k-2} -10^N)F_{m}+(-1)^{k+1}F_{m-k} \nonumber\\
    10^N-1  \ &\ge \ (F_{k+2}-F_{k-2} -10^N)F_{m}+(-1)^{k+1}F_{m-k}  \ \ge \ 0.
    \end{align}
    Thus, $F_{k+2} - F_{k-2} \ge 10^N$; otherwise, \eqref{eq:Ndigit} would not hold.

    If $F_{k+2} - F_{k-2} \ge 10^N+2$, by \eqref{eq:Ndigit}, we obtain that
    $$10^N-1 \ \ge \ 2F_m + (-1)^{k+1}F_{m-k} \ \ge \ F_m.$$
    Hence, any appendable $F_m$ in the walk must be $\le 10^N-1$.

    Now, we remain to examine the cases when $F_{k+2} - F_{k-2}$ is either exactly $10^N$ or $10^N+1$. If $F_{k+2} - F_{k-2} = 10^N$, from Lemma \ref{lem:10Nimpossible}, we know that this case is not possible. Otherwise, if $F_{k+2} - F_{k-2} = 10^N+1$, from \eqref{eq:Ndigit} and $0 \le d \le 10^N-1$, we have
    $$ 10^N-1 \ \ge \ F_m + (-1)^{k+1}F_{m-k} \ \ge \ 0. $$
    If $k$ is odd, $F_{m}$ has to be $\le 10^N-1$ just like the result we have had so far. However, if $k$ is even, we have $ 10^N-1 \ge F_m - F_{m-k} \ge 0$. Therefore,
    \begin{equation}F_m \ \le \ 10^N-1 + F_{m-k},\label{eq:Ndigit2}\end{equation}
    so if we can approximate an upper bound of $F_{m-k}$ in terms of $F_m$, we can find a bound for $F_m$. By our supposition $F_{k+2} - F_{k-2} = 10^N+1$, we have $k \ge 5$ because when $N=1$, $k=5$. Then, $F_{m-k} \le F_{m-5}  \le F_m/8$ by using the bounding technique in Lemma \ref{lem:Fkbound}. Therefore, we have that, $F_{m-k} \le F_m/8$, so, by \eqref{eq:Ndigit2}, $$10^N-1 \ \ge \ 7F_m/8 \quad\text{ or }\quad F_m \ \le \ 8/7\cdot (10^N-1).$$

Thus, since the bound $8/7\cdot (10^N-1)$ is greater than $10^N-1$, we conclude that any appendable step in the walk must be less than $8/7\cdot (10^N-1)$. This implies that any walk on Fibonacci must terminate as soon as the number is greater $8/7\cdot (10^N-1)$ given we append exactly $N$ digits each time.\end{proof}

\begin{corollary}\label{cor:fiboNdigit}
The implication of Lemma \ref{lem:fiboNdigit} is that any appendable step in a walk must contain at most $\lfloor\log(8/7\cdot (10^N-1))+1\rfloor = \lfloor 0.058+N+1\rfloor = N+1$ digits, given we append exactly $N$ digits each time. Since any number not greater than $8/7\cdot (10^N-1)$ will contain at least $N+1$ digits after appended by $N$ digits one time, we can append at most twice.
\end{corollary}

\begin{corollary}\label{cor:fiboNdigit2} For any natural number $M \le N$, Corollary \ref{cor:fiboNdigit} says that we cannot append $M$ digits to the right when starting with a number of $N+2 \ge M+2$ digits because the starting point already has too many digits. In other words, if a Fibonacci number has at least $N+2$ digits, we cannot append $1, 2, \ldots,$ or $N$ digits to the right of that number to obtain another Fibonacci number.
\end{corollary}

\begin{proof}[Proof of Theorem \ref{thm:FiboAtmostN}]
Given that we start with a Fibonacci number $A_0$ that has $N_0$ digits. 
Corollary \ref{cor:fiboNdigit2} then implies that we cannot append $1, \dots, N_0-2$ digits to $A_0$. Thus, we can only append $N_0-1$ digits or above in the first appending. Then, after the first appending, the newly appended number, $A_1$, now contains at least $N_0+N_0-1 = 2N_0-1$ digits. Again, Corollary \ref{cor:fiboNdigit2} implies that we can only append $2N_0-2$ or above number of digits in the second appending.
Repeating the process above, we are required to append at least $2^{M-1}(N_0-1)$ digits at the $M$-th step. Hence, we can determine the largest $M$ as follows:
\begin{eqnarray*}
2^{M-1}(N_0-1) &\le& N \\
M &\le&\log_2{\frac{N}{N_0 -1 }}+1.
\end{eqnarray*}
Therefore, the length of the longest walk is at most $\lfloor\log_2{\frac{N}{N_0-1}}\rfloor+2$, including the starting number. Notice that this formula does not work for $N_0 = 1$ since we do not want to append $N_0-1 = 0$ digit. However, by similar analysis, we obtain that  $\lfloor\log_2{N}\rfloor+2$ is the length of the longest walk starting with a single-digit number.
\end{proof}

By exploiting several relationships among Fibonacci numbers, we conclude that there is no walk to infinity on the Fibonacci sequence, given we append at most $N$ digits at a time to the right. In addition, the length $\lfloor \log_2{\frac{N}{N_0-1}}\rfloor+2$ in Theorem \ref{thm:FiboAtmostN} suggests us that the length of any walk on Fibonacci is relatively small compared to $N$, a fixed positive integer.

\medskip

\noindent MSC2020: 11B39

\end{document}